\documentclass[12pt,twoside,reqno]{amsart}
\usepackage{amsmath}
\usepackage{amsfonts}
\usepackage{amssymb}
\usepackage{color}
\usepackage{mathrsfs}
\usepackage{cite}
\newtheorem{theorem}{Theorem}[section]
\newtheorem{corollary}[theorem]{Corollary}
\newtheorem{lemma}[theorem]{Lemma}
\newtheorem{proposition}[theorem]{Proposition}

\numberwithin{equation}{section}

\newcommand{\tr}[1]{{\color{red}{#1}}}
\begin{document}
\title{Optimal extinction rates for the fast diffusion equation with strong absorption}
\thanks{}
\author{Razvan Gabriel Iagar}
\address{Instituto de Ciencias Matem\'aticas (ICMAT), Nicolas Cabrera 13-15, Campus de Cantoblanco,
E--28049, Madrid, Spain}
\email{razvan.iagar@icmat.es}
\address{Institute of Mathematics of the
Romanian Academy, P.O. Box 1-764, RO-014700, Bucharest, Romania.}

\author{Philippe Lauren\c{c}ot}
\address{Institut de Math\'ematiques de Toulouse, UMR~5219, Universit\'e de Toulouse, CNRS \\ F--31062 Toulouse Cedex 9, France}
\email{laurenco@math.univ-toulouse.fr}


\date{\today}

\begin{abstract}
Optimal extinction rates near the extinction time are derived for non-negative solutions to a fast diffusion equation with strong absorption, the power of the absorption exceeding that of the diffusion.
\end{abstract}

\maketitle

%
%
\pagestyle{myheadings}
\markboth{\sc{R.G. Iagar and Ph. Lauren\c cot}}{\sc{Optimal extinction rates for fast diffusion with absorption}}

\section{Introduction}

Given $m\in (0,\infty)$, $q\in (0,1)$, and a non-negative initial condition $u_0$ in $BC(\mathbb{R}^N)$, $u_0\not\equiv 0$, it is well-known that the initial value problem
\begin{subequations}\label{a1}
\begin{align}
\partial_t u - \Delta u^m + u^q & = 0\ , \qquad (t,x)\in (0,\infty)\times \mathbb{R}^N\ , \label{a1a} \\
u(0) & = u_0\ , \qquad x\in \mathbb{R}^N\ , \label{a1b}
\end{align}
\end{subequations}
has a unique non-negative (weak) solution $u$ which vanishes identically after a finite time, a phenomenon usually referred to as \textit{finite time extinction}\cite{Ka74, Ka87, Ke83}. More precisely, introducing the extinction time
\begin{equation}
T_e := \sup\{ t>0\ :\ u(t)\not\equiv 0\}>0\ , \label{a1c}
\end{equation}
then $T_e$ is finite and satisfies $T_e\le \|u_0\|_\infty^{(1-q)}/(1-q)$, the latter upper bound being a straightforward consequence of \eqref{a1} and the comparison principle. Moreover, there holds
\begin{equation}
u(t)\not\equiv 0 \;\text{ for }\; t\in [0,T_e) \;\text{ and }\; u(t) \equiv 0  \;\text{ for }\; t\ge T_e\ . \label{a2}
\end{equation}
When $q<m$ and $u_0(x)\to 0$ as $|x|\to\infty$, finite time extinction is accompanied by an even more striking phenomenon, the \textit{instantaneous shrinking of the support}, that is, the positivity set $\mathcal{P}(t) := \{ x\in \mathbb{R}^N\ :\ u(t,x)>0\}$ of $u$ at time $t$ is a relatively compact subset of $\mathbb{R}^N$ for all $t\in (0,T_e)$, even if $\mathcal{P}(0)=\mathbb{R}^N$ initially \cite{Ab98, BU94, EK79, Ka74}. Observe that the inequality $q<m$ is always satisfied when the diffusion is linear ($m=1$) or slow ($m>1$). Additional information on the behaviour of $\mathcal{P}(t)$ as $t\to T_e$ is also available when $m\ge 1$ and $N=1$ \cite{CMM95, FP16, FPFR12, FH87, GSV99, GSV00}.

Once finite time extinction is known to take place, gaining further insight into the underlying mechanism requires to identify the behaviour of $u(t)$ as $t\to T_e$, a preliminary step being to determine the relevant space and time scales. Simple scaling arguments predict that, for $r\in [1,\infty]$ and $u_0\in L^r(\mathbb{R}^N)$, there is a constant $\gamma_r>0$ (depending on $N$, $m$, $q$, $u_0$, and $r$) such that
\begin{equation}
\| u(t)\|_r \sim \gamma_r (T_e-t)^{\alpha - (N\beta/r)}\ , \label{a3}
\end{equation}
where
\begin{equation}
\alpha := \frac{1}{1-q}>0\ , \qquad \beta := \frac{q-m}{2(1-q)} \in \mathbb{R}\ . \label{a4}
\end{equation}
As already observed by several authors \cite{FH87, FV01, GV94a}, a rather simple comparison argument provides a lower bound for the $L^\infty$-norm of the form \eqref{a3}. Indeed, consider $t\in (0,T_e)$ and let $x(t)\in\mathbb{R}^N$ be a point where $u(t)$ reaches its maximum value, that is, $u(t,x(t))=\|u(t)\|_\infty$. Then $u(t)^m$ also attains its maximum value at this point, so that $\Delta u^m(t,x(t))\le 0$ and we infer from \eqref{a1a} that (at least formally)
$$
\frac{d}{dt}\|u(t)\|_\infty = \partial_t u(t,x(t)) \le - u(t,x(t))^q = - \|u(t)\|_\infty^q\ .
$$
Integrating the above differential inequality over $(t,T_e)$ gives the expected lower bound
\begin{equation}
\|u(t)\|_\infty \ge [(1-q) (T_e-t)]^{1/(1-q)}\ , \qquad t\in [0,T_e)\ . \label{a5}
\end{equation}
The derivation of an upper bound of the form \eqref{a3} turns out to be more involved and the results obtained so far are rather sparse: in one space dimension, the upper bound
\begin{equation}
\|u(t)\|_\infty \le C_\infty (T_e-t)^{1/(1-q)}\ , \qquad t\in [0,T_e)\ , \label{a6}
\end{equation}
is shown in \cite[Proposition~2.2]{HV92} for $m=1$ and in \cite[Lemma~5.2, Lemma~7.2 \& Lemma~9.2]{FV01} for $m\in (0,1)$, the latter being valid only for compactly supported initial data. The proofs are however of a completely different nature: in \cite{HV92}, properties of the linear heat equation are used while the approach in \cite{FV01} relies on the intersection-comparison technique, which requires in particular the compactness of the support of the initial condition. Still for $m=1$ but in any space dimension, the upper bound \eqref{a6} is derived in \cite[Lemma~2.1]{FH87} for radially symmetric initial data $u_0$ having a non-increasing profile and satisfying $\Delta u_0 + \mu u_0^q\ge 0$ in $\mathbb{R}^N$ for some $\mu>0$. The last case for which \eqref{a6} is proved corresponds to the choice $m=2-q>1$ and the proof relies on the derivation of an Aronson-B\'enilan estimate, which seems to be only available for this specific choice of the parameters $m$ and $q$ \cite{GV94a}.

The purpose of this note is to contribute to the validity of \eqref{a3} and derive optimal upper and lower bounds near the extinction time when the parameters $m$ and $q$ range in
\begin{equation}
\frac{(N-2)_+}{N}<m<q<1\ . \label{a7}
\end{equation}
Recalling that a lower bound in $L^\infty$ is already available, see \eqref{a5}, we begin with upper bounds.

\begin{theorem}[Upper bounds]\label{Thma8}
Assume that $m$ and $q$ satisfy \eqref{a7} and consider a non-negative initial condition $u_0\in BC(\mathbb{R}^N)$, $u_0\not\equiv 0$, for which there is $\kappa_0>0$ such that
\begin{equation}
u_0(x) \le \kappa_0 |x|^{-2/(q-m)}\ , \qquad x\in\mathbb{R}^N\ . \label{a9}
\end{equation}
Given $r\in [1,\infty]$, there is $C_r>0$ depending only on $N$, $m$, $q$, $u_0$, and $r$ such that the solution $u$ to \eqref{a1a}-\eqref{a1b} satisfies
\begin{equation}
\|u(t)\|_r \le C_r (T_e-t)^{\alpha-(N\beta/r)}\ , \qquad t\in (0,T_e)\ , \label{a10}
\end{equation}
the extinction time $T_e$ being defined in \eqref{a1c}.
\end{theorem}

Theorem~\ref{Thma8} thus extends the validity of the upper bound \eqref{a6} established in \cite{FV01} for $N=1$ and $r=\infty$ to any space dimension $N\ge 1$ and $r\in [1,\infty]$, while relaxing the assumption of compact support required in \cite{FV01}. It is worth mentioning that the validity of \eqref{a10} for $r\in [1,\infty)$ does not seem to be a simple consequence of \eqref{a10} for $r=\infty$ since $u(t)$ is positive everywhere in $\mathbb{R}^N$ for all $t\in (0,T_e)$ even if $u_0$ is compactly supported, see \cite[Lemma~2.5]{FV01} and Proposition~\ref{Propa15} below.

To be able to cope with higher space dimensions and non-compactly supported initial data, the proof of Theorem~\ref{Thma8} takes a different route from that in \cite{FV01} and is carried out in two steps: we first show that the algebraic decay at infinity \eqref{a9} enjoyed by $u_0$ remains true throughout time evolution and combine it with \eqref{a1a} to prove \eqref{a10} for $r=1$. We next use self-similar variables and Moser's interation technique to derive \eqref{a10} for all $r\in (1,\infty]$.

As a consequence of \eqref{a5} and Theorem~\ref{Thma8} for $r=\infty$, the correct time scale for the extinction phenomenon is identified. We now supplement the lower bound \eqref{a5} in $L^\infty$ with another one in $L^{m+1}$. On the one hand, it allows us to identify the right space scale. On the other hand, its derivation does not rely on the comparison principle but on energy estimates, a technique which is more likely to extend to other problems for which the former might not be available.

\begin{theorem}[Lower bound in $L^{m+1}$]\label{Thma11}
Assume that $m$ and $q$ satisfy \eqref{a7} and consider a non-negative initial condition $u_0\in BC(\mathbb{R}^N)$, $u_0\not\equiv 0$, such that $u_0\in L^{m+1}(\mathbb{R}^N)$. There is $c_{m+1}>0$ depending only on $N$, $m$, $q$, and $u_0$ such that the solution $u$ to \eqref{a1a}-\eqref{a1b} satisfies
\begin{equation}
\|u(t)\|_{m+1} \ge c_{m+1} (T_e-t)^{\alpha-(N\beta/(m+1))}\ , \qquad t\in (0,T_e)\ , \label{a12}
\end{equation}
the extinction time $T_e$ being defined in \eqref{a1c}.
\end{theorem}

Observing that Theorems~\ref{Thma8} and~\ref{Thma11} are shown without using the $L^\infty$-lower bound \eqref{a5}, the latter may be recovered from these two results by H\"older's inequality, with a less explicit constant though.

\begin{corollary}\label{Cora13}
Assume that $m$ and $q$ satisfy \eqref{a7} and consider a non-negative initial condition $u_0\in BC(\mathbb{R}^N)$, $u_0\not\equiv 0$, enjoying the decay property \eqref{a9}. For $r\in (m+1,\infty]$, there is $c_r>0$ depending only $N$, $m$, $q$, $u_0$, and $r$ such that the solution $u$ to \eqref{a1a}-\eqref{a1b} satisfies
\begin{equation}
\|u(t)\|_r \ge c_r (T_e-t)^{\alpha-(N\beta/r)}\ , \qquad t\in (0,T_e)\ . \label{a14}
\end{equation}
\end{corollary}

Summarizing the outcome of Theorem~\ref{Thma8}, Theorem~\ref{Thma11}, and Corollary~\ref{Cora13}, we have shown that, for all non-negative initial data $u_0\in BC(\mathbb{R}^N)$, $u_0\not\equiv 0$, enjoying the decay property \eqref{a9}, the corresponding solution $u$ to \eqref{a1a}-\eqref{a1b} is bounded in $L^r(\mathbb{R}^N)$, $r\in [m+1,\infty]$, from above and from below at time $t\in (0,T_e)$ by the same power of $T_e-t$. Such estimates pave the way towards a more precise description of the behaviour of $u(t)$ as $t\to T_e$, which is expected to be self-similar. That this is indeed the case is shown in \cite{FGV02, FV01} in one space dimension, another building block of the proof being the uniqueness of self-similar solutions \cite{FGV02}.

We end up this note with the already mentioned everywhere positivity of solutions to \eqref{a1a}-\eqref{a1b} for positive times prior to the extinction time. As we shall see below, this property holds true for a wider range of the parameters $m$ and $q$, namely $0 \tr{<} m \le q <1$. It is already observed in \cite[Lemma~2.5]{FV01} in one space dimension and we extend it herein to any space dimension. It is worth emphasizing that it includes the case $q=m$ and contrasts markedly with the instantaneous shrinking of the support occurring when $q<m$.

\begin{proposition}[Everywhere positivity]\label{Propa15}
Consider $0< m \le q <1$. Let $u_0\in  BC(\mathbb{R}^N)$ be a non-negative initial condition, $u_0\not\equiv 0$, and denote the corresponding solution to \eqref{a1a}-\eqref{a1b} by $u$ with extinction time $T_e$. For $t\in (0,T_e)$, there holds
\begin{equation}
\mathcal{P}(t) := \{ x\in \mathbb{R}^N\ :\ u(t,x)>0\} = \mathbb{R}^N\ . \label{a16}
\end{equation}
\end{proposition}
Before proving the results stated above, we point out once more that the energy techniques developed herein seem to be rather flexible and are expected to have a wider range of applicability. For instance, a related approach is used in the companion paper \cite{IL17}, where optimal (lower and upper) bounds near the extinction time are established for a different fast diffusion equation (featuring the $p$-Laplacian operator, $p\in (1,2)$) with a gradient absorption term.

\section{Upper bounds near the extinction time}

Throughout this section, we assume that $m$ and $q$ satisfy \eqref{a7} and consider a non-negative initial condition $u_0\in BC(\mathbb{R}^N)$, $u_0\not\equiv 0$, enjoying the decay property \eqref{a9}. Let $u$ be the corresponding solution to \eqref{a1a}-\eqref{a1b}.

\subsection{$L^1$-estimate}

We begin with the propagation throughout time evolution of the algebraic decay \eqref{a9} and set
\begin{equation}
\kappa_* := \left( \frac{2m(m+q)}{(q-m)^2} \right)^{1/(q-m)}\ . \label{b1}
\end{equation}

\begin{lemma}\label{Lemb2} For $t\in [0,\infty)$ and $x\in\mathbb{R}^N\setminus\{0\}$, there holds
$$
u(t,x) \le \max\{\kappa_0 , \kappa_*\} |x|^{-2/(q-m)}\ .
$$
\end{lemma}

\begin{proof}
Set $\Sigma_\kappa(x) := \kappa |x|^{-2/(q-m)}$ for $x\in\mathbb{R}^N\setminus\{0\}$, where $\kappa$ is a positive constant yet to be determined. We note that
\begin{align*}
-\Delta \Sigma^m(x) + \Sigma(x)^q & = - \kappa^m \left[ \frac{2m(m+q)}{(q-m)^2} |x|^{-2q/(q-m)} - \frac{2m(N-1)}{(q-m)} |x|^{-2q/(q-m)} \right] \\
& \quad + \kappa^q |x|^{-2q/(q-m)} \\
& \ge \kappa^m \left( \kappa^{q-m} - \kappa_*^{q-m} \right) |x|^{-2q/(q-m)}
\end{align*}
for $x\in \mathbb{R}^N\setminus\{0\}$, so that $\Sigma_\kappa$ is a supersolution to \eqref{a1a} in $\mathbb{R}^N\setminus\{0\}$ for all $\kappa\ge \kappa_*$. We then choose $\kappa = \max\{\kappa_0 , \kappa_*\}$ and use the comparison principle to complete the proof of Lemma~\ref{Lemb2}.
\end{proof}

We are now in a position to derive the claimed upper bound near the extinction time for $r=1$.

\begin{proof}[Proof of Theorem~\ref{Thma8}: $r=1$]
Let $t\in [0,T_e)$. Integrating \eqref{a1a} over $(t,T_e)\times\mathbb{R}^N$ gives
\begin{equation}
\|u(t)\|_1 = \int_t^{T_e} \int_{\mathbb{R}^N} u(s,x)^q\ dxds\ . \label{b3}
\end{equation}
Owing to \eqref{a7}, there holds $2q/(q-m)>N$ and we infer from Lemma~\ref{Lemb2} and H\"older's inequality that, for $s\in (t,T_e)$ and $R>0$,
\begin{align*}
\int_{\mathbb{R}^N} u(s,x)^q\ dx & = \int_{B(0,R)} u(s,x)^q\ dx + \int_{\mathbb{R}^N\setminus B(0,R)} u(s,x)^q\ dx \\
& \le C R^{N(1-q)} \|u(s)\|_1^q + (\max\{\kappa_0,\kappa_*\})^q |\mathbb{S}^{N-1}| \int_R^\infty r^{N-1-(2q/(q-m))}\ dr \\
& \le C \left( R^{N(1-q)} \|u(s)\|_1^q + R^{(N(q-m)-2q)/(q-m)} \right)\ .
\end{align*}
We next optimize in $R$ in the previous inequality by setting $R(s) := \|u(s)\|_1^{-(q-m)/(N(m-q)+2)}$, which satisfies
$$
R(s)^{N(1-q)} \|u(s)\|_1^q = R(s)^{(N(q-m)-2q)/(q-m)} = \|u(s)\|_1^{(N(m-q)+2q)/(N(m-q)+2)}\ .
$$
Consequently, taking $R=R(s)$ in the previous inequality, we obtain
$$
\int_{\mathbb{R}^N} u(s,x)^q\ dx \le C \|u(s)\|_1^{(N(m-q)+2q)/(N(m-q)+2)}\ ,
$$
which gives, together with \eqref{b3}, the positivity of $N(m-q)+2q$, and the time monotonicity of $s\mapsto \|u(s)\|_1$,
\begin{align*}
\|u(t)\|_1 & \le C \int_t^{T_e} \|u(s)\|_1^{(N(m-q)+2q)/(N(m-q)+2)}\ ds \\
& \le C (T_e-t) \|u(t)\|_1^{(N(m-q)+2q)/(N(m-q)+2)}\ ,
\end{align*}
from which \eqref{a10} for $r=1$ readily follows.
\end{proof}

\subsection{Scaling variables and $L^r$-estimates, $r\in (1,\infty]$}

The next step is to take advantage of the just derived $L^1$-upper bound to derive the corresponding ones in $L^r$ for $r\in (1,\infty]$. To this end, we introduce the scaling variables
\begin{equation}
s := \ln\left( \frac{T_e}{T_e-t} \right)\ , \qquad y := x (T_e-t)^\beta\ , \qquad (t,x)\in [0,T_e)\times \mathbb{R}^N\ , \label{b4}
\end{equation}
and the new unknown function $v$ defined by
\begin{equation}
u(t,x) = (T_e-t)^\alpha v\left( \ln(T_e) - \ln(T_e-t) , x (T_e-t)^\beta \right)\ , \qquad (t,x)\in [0,T_e)\times \mathbb{R}^N\ , \label{b5}
\end{equation}
or, equivalently,
\begin{equation}
v(s,y) = T_e^{-\alpha} e^{\alpha s} u\left( T_e (1-e^{-s}) , y T_e^{-\beta} e^{\beta s} \right)\ , \qquad (s,y)\in [0,\infty)\times \mathbb{R}^N\ . \label{b6}
\end{equation}
It readily follows from \eqref{a1a}-\eqref{a1b} that $v$ solves
\begin{align}
\partial_s v(s,y) & = \alpha v(s,y) + \beta y\cdot \nabla v(s,y) + \Delta v^m(s,y) - v(s,y)^q\ , \qquad (s,y)\in (0,\infty)\times \mathbb{R}^N\ , \label{b7a} \\
v(0,y) & = v_0(y) := T_e^{-\alpha} u_0\left( y T_e^{-\beta} \right)\ , \qquad y\in \mathbb{R}^N\ . \label{b7b}
\end{align}
Since
\begin{equation}
\|u(t)\|_r = (T_e-t)^{\alpha-(N\beta/r)} \|v(s)\|_r\ , \qquad t\in (0,T_e)\ , \label{b8}
\end{equation}
for all $r\in [1,\infty]$, we realize that an upper bound such as \eqref{a10} on $\|u(t)\|_r$ for $t\in (0,T_e)$ obviously follows from a uniform upper bound on $\|v(s)\|_r$ for $s\ge 0$, the converse being true as well. In particular, it follows from \eqref{b8} and Theorem~\ref{Thma8} for $r=1$ that
\begin{equation}
\|v(s)\|_1 \le C_1\ , \qquad s\ge 0\ , \label{b9}
\end{equation}
and we may assume without loss of generality that $C_1\ge 1$.

We now aim at using a bootstrap argument to deduce from \eqref{b7a} and \eqref{b9} that $v$ belongs to $L^\infty(0,\infty;L^r(\mathbb{R}^N))$ for all $r\in (1,\infty]$. To this end, Moser's iteration technique is a suitable tool and the way we apply it is inspired from \cite[Theorem~3.1]{Al79}. But since \cite[Theorem~3.1]{Al79} is devoted to the slow diffusion case $m>1$, some technical aspects of its proof do not seem to apply directly here and we borrow additional arguments from the proof of \cite[Proposition~2]{AMST98}.

\begin{lemma}\label{Lemb10}
Let $r\in (0,\infty]$. There is $C_{r+1}>0$ depending only on $N$, $m$, $q$, $u_0$, and $r$ such that
$$
\|v(s)\|_{r+1} \le C_{r+1}\ , \qquad s\ge 0\ .
$$
\end{lemma}

\begin{proof}
Let $r\in [2-m,\infty)$. Multiplying \eqref{b7a} by $v^r$, integrating over $\mathbb{R}^N$, and using integration by parts, we obtain
\begin{align*}
\frac{1}{r+1} \frac{d}{ds}\|v\|_{r+1}^{r+1} + rm \int_{\mathbb{R}^N} v^{r+m-2} |\nabla v|^2\ dy + \int_{\mathbb{R}^N} v^{r+q}\ dy & = \left( \alpha - \frac{N\beta}{r+1} \right)\|v\|_{r+1}^{r+1}\ , \\
\frac{d}{ds}\|v\|_{r+1}^{r+1} + \frac{4mr(r+1)}{(m+r)^2} \left\| \nabla v^{(m+r)/2} \right\|_2^2 & \le \alpha (r+1) \|v\|_{r+1}^{r+1}\ .
\end{align*}
Since $4m r(r+1) \ge 2m (m+r)^2$, we end up with
\begin{equation}
\frac{d}{ds}\|v\|_{r+1}^{r+1} + 2m \left\| \nabla v^{(m+r)/2} \right\|_2^2 \le \alpha (r+1) \|v\|_{r+1}^{r+1}\ . \label{b11}
\end{equation}

We next fix $\zeta\in (2/m,2^*)$ where $2^* := 2N/(N-2)_+$ (with $2^*=\infty$ for $N=1,2$). On the one hand, it follows from the Gagliardo-Nirenberg inequality that
\begin{equation}
\left\| v^{(m+r)/2} \right\|_\zeta \le C \left\| \nabla v^{(m+r)/2} \right\|_2^\theta \left\| v^{(m+r)/2} \right\|_1^{1-\theta}\ , \label{b12}
\end{equation}
with
$$
\theta := \frac{2N(\zeta-1)}{(N+2)\zeta}\ .
$$
On the other hand, since $(m+r)/2\in [1,\zeta(m+r)/2]$ for $r\ge 2-m$, we infer from H\"older's inequality that
\begin{equation}
\| v \|_{(m+r)/2}^{(m+r)/2} \le \|v\|_{\zeta(m+r)/2}^{\zeta(m+r)(m+r-2)/2[\zeta(m+r)-2]} \|v\|_1^{(\zeta-1)(m+r)/[\zeta(m+r)-2]}\ . \label{b13}
\end{equation}
We deduce from \eqref{b12} and \eqref{b13} that
\begin{align*}
\|v\|_{\zeta(m+r)/2}^{(m+r)/2} & = \left\| v^{(m+r)/2} \right\|_\zeta\leq C\|\nabla v^{(m+r)/2}\|_2^{\theta}\left(\|v\|_{(m+r)/2}^{(m+r)/2}\right)^{1-\theta} \\
& \le C \left\| \nabla v^{(m+r)/2} \right\|_2^\theta \left[ \|v\|_{\zeta(m+r)/2}^{\zeta(m+r)(m+r-2)/2[\zeta(m+r)-2]} \|v\|_1^{(\zeta-1)(m+r)/[\zeta(m+r)-2]}\right]^{1-\theta}\ ,
\end{align*}
hence
\begin{equation}
\|v\|_{\zeta(m+r)/2}^{\zeta(m+r)[N(m+r)+2-N]/N[\zeta(m+r)-2]} \le C \left\| \nabla v^{(m+r)/2} \right\|_2^2 \|v\|_1^{[2N-(N-2)\zeta](m+r)/N[\zeta(m+r)-2]}\ . \label{b14}
\end{equation}
Moreover, since $\zeta>2/m$ and $m<1$, we have $2r\leq m[\zeta(m+r)-2]$, hence
\begin{align*}
\frac{[2N-(N-2)\zeta](m+r)}{N[\zeta(m+r)-2]} & \le \frac{[2N-(N-2)\zeta]}{N} \frac{m(m+r)}{2r} \\
& \le \frac{m[2N-(N-2)\zeta]}{N}\ ,
\end{align*}
so that
\begin{align}
\|v\|_1^{[2N-(N-2)\zeta](m+r)/N[\zeta(m+r)-2]} & \le C_1^{[2N-(N-2)\zeta](m+r)/N[\zeta(m+r)-2]} \nonumber \\
& \le C_1^{m[2N-(N-2)\zeta]/N}\ . \label{b15}
\end{align}
Also,
$$
1 - \frac{N [\zeta(m+r)-2]}{\zeta [N(m+r)+2-N]} = \frac{2N-\zeta(N-2)}{\zeta [N(m+r)+2-N]}>0\ ,
$$
and we infer from \eqref{b14}, \eqref{b15}, and Young's inequality that
\begin{align*}
\| v\|_{\zeta(m+r)/2}^{m+r} & \le \frac{N [\zeta(m+r)-2]}{\zeta [N(m+r)+2-N]} \|v\|_{\zeta(m+r)/2}^{\zeta(m+r)[N(m+r)+2-N]/N[\zeta(m+r)-2]} \\
& \quad + \frac{2N-\zeta(N-2)}{\zeta [N(m+r)+2-N]} \\
& \le C \left\| \nabla v^{(m+r)/2} \right\|_2^2 + 1 \ .
\end{align*}
Therefore, there is $\nu\in (0,1)$ depending only on $N$, $m$, $q$, and $u_0$ such that
\begin{equation}
\nu \left( \| v\|_{\zeta(m+r)/2}^{m+r} - 1 \right) \le \left\| \nabla v^{(m+r)/2} \right\|_2^2\ . \label{b16a}
\end{equation}
Moreover, since $r+1\in [1,\zeta(m+r)/2]$, it follows from \eqref{b9} and H\"older's and Young's inequalities that
\begin{align}
\|v\|_{r+1}^{r+1} & \le \|v\|_{\zeta(m+r)/2}^{\zeta r(r+m)/[\zeta(r+m)-2]} \|v\|_1^{[(\zeta-2)r+\zeta m -2]/(\zeta r + \zeta m - 2)} \nonumber\\
& \le C_1^{[(\zeta-2)r+\zeta m -2]/(\zeta r + \zeta m - 2)} \|v\|_{\zeta(m+r)/2}^{\zeta r(r+m)/[\zeta(r+m)-2]} \nonumber\\
& \le C_1 \|v\|_{\zeta(m+r)/2}^{\zeta r(r+m)/[\zeta(r+m)-2]} \nonumber\\
& \le \frac{\zeta m - 2}{\zeta(m+r)-2} C_1^{[\zeta(m+r)-2]/(\zeta m-2)} + \frac{\zeta r}{\zeta(m+r)-2} \|v\|_{\zeta(m+r)/2}^{m+r} \nonumber\\
& \le C_1^{[\zeta(m+r)-2]/(\zeta m-2)} + \|v\|_{\zeta(m+r)/2}^{m+r} \label{b16b}
\end{align}

Next, let $\sigma>1$ to be chosen appropriately later on and set
$$
\mathcal{I}_r(s) := \int_{\mathbb{R}^N} v(s,y)^{\zeta[r+m+\sigma(1-m)]/[\sigma(\zeta-2)+2]}\ dy\ , \qquad s\ge 0\ .
$$
Since $\sigma(\zeta-2)+2 \in [\zeta,\sigma\zeta]$ and
$$
r+1 = \frac{\sigma-1}{\sigma} (m+r) + \frac{m+r+\sigma(1-m)}{\sigma}\ ,
$$
we deduce from \eqref{b9} and  H\"older's and Young's inequalities that, for $\delta>0$,
\begin{align}
\|v\|_{r+1}^{r+1} & \le \|v\|_{\zeta(m+r)/2}^{(\sigma-1)(m+r)/\sigma} \mathcal{I}_r^{[\sigma(\zeta-2)+2]/\sigma\zeta} \nonumber\\
& \le \frac{(\sigma-1)\delta}{\sigma} \|v\|_{\zeta(m+r)/2}^{m+r} + \frac{1}{\sigma \delta^{\sigma-1}} \mathcal{I}_r^{[\sigma(\zeta-2)+2]/\zeta} \nonumber \\
& \le \delta \|v\|_{\zeta(m+r)/2}^{m+r} + \frac{1}{\delta^{\sigma-1}} \mathcal{I}_r^{[\sigma(\zeta-2)+2]/\zeta}\ . \label{b17}
\end{align}
Combining \eqref{b11}, \eqref{b16a}, and \eqref{b17} leads us to
\begin{align*}
\frac{d}{ds}\|v\|_{r+1}^{r+1} & + 2m \nu \left( \| v\|_{\zeta(m+r)/2}^{m+r} - 1 \right) \le \frac{d}{ds}\|v\|_{r+1}^{r+1} + 2m \left\| \nabla v^{(m+r)/2} \right\|_2^2 \\
& \le \alpha (r+1) \|v\|_{r+1}^{r+1} \\
& \le \alpha \delta (r+1) \|v\|_{\zeta(m+r)/2}^{m+r} + \frac{\alpha (r+1)}{\delta^{\sigma-1}} \mathcal{I}_r^{[\sigma(\zeta-2)+2]/\zeta}\ .
\end{align*}
We then choose $\delta = m\nu/\alpha(r+1)$ in the above inequality to obtain
\begin{align*}
\frac{d}{ds}\|v\|_{r+1}^{r+1} + m\nu \| v\|_{\zeta(m+r)/2}^{m+r} \le 2m \nu + \frac{\alpha^\sigma (r+1)^\sigma}{(m\nu)^{\sigma-1}} \mathcal{I}_r^{[\sigma(\zeta-2)+2]/\zeta}\ .
\end{align*}
We finally use \eqref{b16b} to estimate from below the second term of the left-hand side of the previous inequality and end up with
\begin{equation}
\frac{d}{ds}\|v\|_{r+1}^{r+1} + m\nu \| v\|_{r+1}^{r+1} \le 2m \nu + m \nu C_1^{[\zeta(m+r)-2]/(\zeta m-2)} + \frac{\alpha^\sigma (r+1)^\sigma}{(m\nu)^{\sigma-1}} \mathcal{I}_r^{[\sigma(\zeta-2)+2]/\zeta}\ . \label{b18}
\end{equation}
We first choose
$$
\sigma = \frac{\zeta(m+r)-2}{\zeta m -2}>1
$$
in \eqref{b18} and observe that this choice guarantees that
$$
\zeta [m+r+\sigma(1-m)] = \sigma (\zeta-2)+2\ .
$$
Consequently, $\mathcal{I}_r = \|v\|_1$ and we deduce from \eqref{b9} and \eqref{b18} that there is $C(r)>0$ depending on $N$, $m$, $q$, $u_0$, and $r$ such that
$$
\frac{d}{ds}\|v\|_{r+1}^{r+1} + m\nu \| v\|_{r+1}^{r+1} \le C(r)\ .
$$
Integrating the previous differential inequality entails that
\begin{equation}
\sup_{s\ge 0} \|v(s)\|_{r+1} < \infty\ . \label{b19}
\end{equation}
The validity of \eqref{b19} extends to all $r\in (0,2-m)$ by \eqref{b9} and H\"older's inequality.

To complete the proof of Lemma~\ref{Lemb10}, we are left to check the boundedness of $v$ in $L^\infty(\mathbb{R}^N)$. To this end, we take $\sigma = \sigma_0 := 2(\zeta-1)/(\zeta-2)>1$ in \eqref{b18} and obtain, after integration with respect to time,
\begin{align*}
\|v(s)\|_{r+1}^{r+1} & \le \|v_0\|_{r+1}^{r+1} e^{-m\nu s} + 2 + C_1^{[\zeta(m+r)-2]/(\zeta m-2)} \\
& \quad + \left( \frac{\alpha (r+1)}{m\nu} \right)^{\sigma_0} \left[ \sup_{s_*\in [0,s]} \mathcal{I}_r(s_*) \right]^2 \\
& \le \|v_0\|_1 \|v_0\|_\infty^r + 2 + C_1^{[\zeta(m+r)-2]/(\zeta m-2)} \\
& \quad + \left( \frac{\alpha (r+1)}{m\nu} \right)^{\sigma_0} \left[ \sup_{s_*\in [0,s]} \mathcal{I}_r(s_*) \right]^2\ ,
\end{align*}
and
$$
\mathcal{I}_r = \|v\|_{[(r+m)+\sigma_0(1-m)]/2}^{[(r+m)+\sigma_0(1-m)]/2}\ .
$$
Therefore, there are $K_0>0$ and $K_1>0$ depending only on $N$, $m$, $q$, and $u_0$ such that
\begin{equation}
\sup_{s\ge 0}\left\{ \|v(s)\|_{r+1}^{r+1} \right\} \le K_0 \left( K_1^{r+1} + (1+r)^{\sigma_0} \sup_{s\ge 0} \left\{ \|v(s)\|_{[(r+m)+\sigma_0(1-m)]/2}^{[(r+m)+\sigma_0(1-m)]} \right\} \right)\ . \label{b20}
\end{equation}
We now define the sequence $(r_j)_{j\ge 0}$ by
$$
1+r_{j+1} = 2 (1+r_j) - (1 - m)(\sigma_0-1)\ , \qquad j\ge 0\ , \qquad r_0 := 2-m\ ,
$$
and set
$$
V_j := \sup_{s\ge 0}\left\{ \|v(s)\|_{r_j+1}^{r_j+1} \right\} \ , \qquad j\ge 0\ .
$$
For $j\ge 0$, we take $r=r_{j+1}$ in \eqref{b20} and realize that
\begin{align*}
V_{j+1} & \le K_0 \left( K_1^{1+r_{j+1}} + (1+r_{j+1})^{\sigma_0} V_j^2 \right) \\
& \le K_0 (1+r_{j+1})^{\sigma_0} \max\left\{ K_1^{1+r_{j+1}} , V_j^2 \right\}\ , \qquad j\ge 0\ .
\end{align*}
Since $\sigma_0-1<1/(1-m)$ thanks to the constraint $\zeta>2/m$, one has $1+r_0 - (1-m)(\sigma_0-1)>0$ and we are in a position to apply \cite[Lemma~A.1]{La94}, which we recall in Lemma~\ref{Lemb100} below for completeness, to conclude that there is $K_2>0$ depending only on $m$, $\zeta$, $K_0$, and $K_1$ such that
$$
V_j^{1/(1+r_j)} \le K_2\ , \qquad j\ge 0\ .
$$
Equivalently,
$$
\sup_{s\ge 0} \{\|v(s)\|_{1+r_j}\} \le K_2\ , \qquad j\ge 0\ ,
$$
and letting $j\to\infty$ entails that $\|v(s)\|_\infty \le K_2$ for all $s\ge 0$, thereby completing the proof of Lemma~\ref{Lemb10}.
\end{proof}

The proof of Theorem~\ref{Thma8} for $r\in (1,\infty]$ is now a straightforward consequence of \eqref{b8} and Lemma~\ref{Lemb10}.

\begin{lemma}\label{Lemb100}
Let $a>1$, $b\ge 0$, $c\in\mathbb{R}$, $C_0\ge 1$, $C_1\ge 1$, and $p_0>0$ be given such that $p_0 (a-1) + c >0$. We define the sequence $(p_k)_{k\ge 0}$ of positive real numbers by $p_{k+1} = a p_k + c$ for $k\ge 0$ and assume that $(Q_k)_{k\ge 0}$  is a sequence of positive real numbers satisfying
$$
Q_0\le C_1^{p_0}\ , \qquad Q_{k+1} \le C_0 p_{k+1}^b \max\left\{ C_1^{p_{k+1}} , Q_k^a \right\}, \ k\ge 0\ .
$$
Then the sequence $\left( Q_k^{1/p_k} \right)_{k\ge 0}$ is bounded.
\end{lemma}

\section{Lower bound near the extinction time}

We now turn to the lower bound near the extinction time in $L^{m+1}(\mathbb{R}^N)$.

\begin{proof}[Proof of Theorem~\ref{Thma11}] For $t\in [0,T_e]$, we define
$$
X(t) := \|u(t)\|_{m+1}^{m+1} \;\;\text{ and }\;\; Y(t) := \int_{\mathbb{R}^N} u(t,x)^{m+q}\ dx\ .
$$
Let $t\in (0,T_e)$. It follows from \eqref{a1a} that
\begin{equation}
\frac{1}{m+1} \frac{dX}{dt}(t) + \|\nabla u^m(t)\|_2^2 + Y(t) = 0\ . \label{c1}
\end{equation}
Since
$$
1 < \frac{m+q}{m} < \frac{m+1}{m} < 2^* := \frac{2N}{(N-2)_+}
$$
by \eqref{a7} we infer from the Gagliardo-Nirenberg inequality that
\begin{align*}
X(t)^{m/(m+1)} & = \left\| u(t)^m \right\|_{(m+1)/m} \le C \left\| \nabla u^m(t) \right\|_2^\theta \left\| u(t)^m \right\|_{(m+q)/m}^{1-\theta} \\
& \le C Y(t)^{m(1-\theta)/(m+q)} \left\| \nabla u^m(t) \right\|_2^\theta\ ,
\end{align*}
where
$$
\theta := \frac{2Nm(1-q)}{(m+1) [m(N+2)-q(N-2)]}\ .
$$
Consequently, since $u(t)\not\equiv 0$ as $t\in (0,T_e)$,
$$
\left\| \nabla u^m(t) \right\|_2^2 \ge C X(t)^{2m/\theta(m+1)} Y(t)^{-2m(1-\theta)/\theta(m+q)}\ ,
$$
which gives, together with \eqref{c1},
\begin{equation}
\frac{dX}{dt}(t) + C X(t)^{2m/\theta(m+1)} Y(t)^{-2m(1-\theta)/\theta(m+q)} + (m+1) Y(t) \le 0\ . \label{c2}
\end{equation}
Setting
$$
\xi := 1 + \frac{2m(1-\theta)}{\theta(m+q)}>1 \;\;\text{ and }\;\; \gamma := \frac{2m}{\theta\xi(m+1)}\ ,
$$
it follows from Young's inequality that
\begin{align*}
X(t)^\gamma & = X(t)^\gamma Y(t)^{-(\xi-1)/\xi} Y(t)^{(\xi-1)/\xi} \le \frac{1}{\xi} X(t)^{\xi\gamma} Y(t)^{1-\xi} + \frac{\xi-1}{\xi} Y(t) \\
& \le X(t)^{2m/\theta (m+1)} Y(t)^{-2m(1-\theta)/\theta (m+q)} + Y(t)
\end{align*}
Combining this inequality with \eqref{c2} leads us to the differential inequality
\begin{equation}
\frac{dX}{dt}(t) + C X(t)^\gamma \le 0\ , \qquad t\in (0,T_e)\ . \label{c3}
\end{equation}
Now,
\begin{align*}
\gamma&=\frac{2m}{\theta\xi(m+1)}=\frac{2m(m+q)}{(2m+\theta(q-m))(m+1)}\\
&=\frac{2m(m+q)[m(N+2)-q(N-2)]}{2m\{(m+1)[m(N+2)-q(N-2)]+N(q-m)(1-q)\}}\\
&=\frac{m(N+2)-q(N-2)}{m(N+2)-qN+2}\in (0,1)\ ,
\end{align*}
and we integrate \eqref{c3} over $(t,T_e)$ to obtain
$$
-X(t)^{1-\gamma} + (1-\gamma) C (T_e-t) \le 0\ , \qquad t\in (0,T_e)\ .
$$
Noticing that
$$
(m+1)\alpha-N\beta=\frac{m+1}{1-q}-\frac{N(q-m)}{2(1-q)}=\frac{m(N+2)-qN+2}{2(1-q)}=\frac{1}{1-\gamma}\ ,
$$
the lower bound \eqref{a12} readily follows from the previous inequality.
\end{proof}

We end up this section with the derivation of the lower bound for $r\in (m+1,\infty]$ from Theorem~\ref{Thma8} for $r=1$ and Theorem~\ref{Thma11}.

\begin{proof}[Proof of Corollary~\ref{Cora13}]
We first note that, owing to \eqref{a7}, there holds $2/(q-m)>N$ and \eqref{a9} entails that $u_0\in L^1(\mathbb{R}^N)$. Since $u_0$ also belongs to $L^\infty(\mathbb{R}^N)$, we conclude that $u_0\in L^{m+1}(\mathbb{R}^N)$.

Let $r\in (m+1,\infty]$ and $t\in (0,T_e)$. We infer from Theorem~\ref{Thma8}, Theorem~\ref{Thma11}, and H\"older's inequality that
\begin{align*}
c_{m+1}^{m+1} (T_e-t)^{(m+1)\alpha - N\beta} & \le \|u(t)\|_{m+1}^{m+1} \le \|u(t)\|_r^{rm/(r-1)} \|u(t)\|_1^{(r-1-m)/(r-1)} \\
& \le C_1 (T_e-t)^{(\alpha-N\beta)(r-1-m)/(r-1)} \|u(t)\|_r^{rm/(r-1)}\ ,
\end{align*}
from which \eqref{a14} readily follows.
\end{proof}

\section{Everywhere positivity}

In this section, we assume that $0<m\le q <1$ and consider a non-negative initial condition $u_0\in BC(\mathbb{R}^N)$, $u_0\not\equiv 0$. We denote the corresponding solution to \eqref{a1a}-\eqref{a1b} by $u$ and define its extinction time by \eqref{a1c}. As in \cite{FV01}, the proof relies on an upper bound for $\partial_t u$ which we establish now.

\begin{lemma}\label{Lemd1}
For $t>0$ there holds
$$
\partial_t u(t) \le \frac{u(t)}{(1-m) t} \;\;\text{ in }\;\; \mathbb{R}^N\ .
$$
\end{lemma}

When $q=m$, Lemma~\ref{Lemd1} is a consequence of \cite[Theorem~2]{BC81}, the proof relying on an homogeneity argument. Though the operator $-\Delta u^m + u^q$ is not homogeneous, we may still adapt the proof of \cite[Theorem~2]{BC81} when $q\ge m$.

\begin{proof}
Given a non-negative initial condition $u_0\in BC(\mathbb{R}^N)$, we denote the corresponding solution to \eqref{a1a}-\eqref{a1b} at time $t\ge 0$ by $S(t) u_0$.  Recall that, if $u_0$ and $v_0$ are two non-negative functions in $BC(\mathbb{R}^N)$ satisfying $u_0\ge v_0$, then the comparison principle entails $S(t) u_0\ge S(t) v_0$ for all $t\ge 0$.

\smallskip

\noindent\textbf{Step~1.} We first claim that, for $\lambda\ge 1$,
\begin{equation}
S(\lambda t) u_0 \le \lambda^{1/(1-m)} S(t)\left( \lambda^{1/(m-1)} u_0 \right)\ , \qquad t\ge 0\ . \label{d2}
\end{equation}
Indeed, setting $u(t) := S(t) u_0$ for $t\ge 0$, the function $v$ defined by $v(t) := \lambda^{1/(m-1)} S(\lambda t) u_0$ satisfies
\begin{align*}
\partial_t v(t,x) - \Delta v^m(t,x) + v(t,x)^q & = \lambda^{m/(m-1)} \partial_t u(\lambda t,x) - \lambda^{m/(m-1)} \Delta u^m(\lambda t,x) \\
& \quad + \lambda^{q/(m-1)} u(\lambda t,x)^q \\
& = \left( \lambda^{q/(m-1)} - \lambda^{m/(m-1)} \right) u(\lambda t,x)^q \le 0\ .
\end{align*}
Since $v(0)=\lambda^{1/(m-1)} u_0 \le u_0$, we infer from the comparison principle that \eqref{d2} holds true.

\smallskip

\noindent\textbf{Step~2.} Now, fix $t>0$ and consider $h>0$. Since $\lambda = (1+h/t)>1$ and $m\in (0,1)$, we infer from \eqref{d2} and the comparison principle that
\begin{align*}
S(t+h) u_0 - S(t) u_0 & = S(\lambda t) u_0 - S(t) u_0 \\
& \le \lambda^{1/(1-m)} S(t)\left( \lambda^{1/(m-1)} u_0 \right) - S(t) u_0 \\
& \le \left[ \left( 1 + \frac{h}{t} \right)^{1/(1-m)} - 1 \right] S(t) u_0\ .
\end{align*}
Dividing the above inequality by $h$ and passing to the limit as $h\to 0$ complete the proof.
\end{proof}

We now argue as in the proof of \cite[Lemma~2.5]{FV01} to complete the proof of Proposition~\ref{Propa15}.

\begin{proof}[Proof of Proposition~\ref{Propa15}]
Fix $t\in (0,T_e)$ and assume for contradiction that $u(t,x_0)=0$ for some $x_0\in\mathbb{R}^N$. By \eqref{a1a} and Lemma~\ref{Lemd1}, there holds
$$
- \Delta u^m(t) + u(t)^q + \frac{u(t)}{(1-m)t} \ge 0 \;\;\text{ in }\;\; \mathbb{R}^N\ ,
$$
so that $u(t)^m$ is a supersolution to
$$
- \Delta w + d w = 0 \;\;\text{ in }\;\; \mathbb{R}^N\ ,
$$
with $d(x) := u(t,x)^{q-m} + u(t,x)^{1-m}/((1-m)t)$ for $x\in \mathbb{R}^N$. Since $t>0$ and $m\le q<1$, the function $d$ is non-negative and bounded and we infer from the strong maximum principle \cite[Theorem~8.19]{GT01} that $u(t)^m\equiv 0$ in $\mathbb{R}^N$, contradicting $t<T_e$. Consequently, $u(t)^m$ is positive everywhere in $\mathbb{R}^N$ and the proof of Proposition~\ref{Propa15} is complete.
\end{proof}

\section*{Acknowledgments} The first author is supported by
the ERC Starting Grant GEOFLUIDS 633152. Part of this
work was done while the first author enjoyed the hospitality and
support of the Institute de Math\'ematiques de Toulouse, Toulouse,
France.

\bibliographystyle{siam}
\bibliography{OptExtRatesFDEAbs}

\end{document}